\documentclass[reqno,a4paper,11pt]{amsart}

\usepackage{amsmath,amstext,amsfonts,amsbsy,eucal,amssymb}
\usepackage[latin1]{inputenc}

\numberwithin{equation}{section}

\newtheorem{theorem}{Theorem}[section]
\newtheorem{lemma}[theorem]{Lemma}

\newtheorem{proposition}[theorem]{Proposition}

\theoremstyle{definition}

\newtheorem{example}[theorem]{Example}
\newtheorem{remark}[theorem]{Remark}

\begin{document}

\parskip 4pt
\baselineskip 16pt


\title[On some generalizations of the sum of powers of natural numbers]
{On some generalizations of the sum of powers of natural numbers}

\author[Andrei K. Svinin]{Andrei K. Svinin}
\address{Andrei K. Svinin, 
Matrosov Institute for System Dynamics and Control Theory of 
Siberian Branch of Russian Academy of Sciences,
P.O. Box 292, 664033 Irkutsk, Russia}
\email{svinin@icc.ru}

%
%

\date{\today}



\begin{abstract}
In this paper   some generalizations of the sum of powers of natural numbers is considered. In particular,  the class of sums whose generating function is the power of the generating function for the classical sums of powers is studying. The so-called binomial sums are also considered. The problem of constructing polynomials that allow  to calculate the values of the corresponding sums in certain cases  is solved.
\end{abstract}

\maketitle

\section{Introduction}

Power-sum  
\begin{equation}
S_m(n):=\sum_{q=1}^{n}q^{m},\;\; m\geq 0
\label{power-sums}
\end{equation}
starting with the works of Blaise Pascal, Johann Faulhaber, Jacob Bernoulli for centuries, has been the object of active research. The expression (\ref{power-sums}) defines a function on the set of natural numbers. Pascal's theorem asserts that for any integer degree $m\geq 0$, the value of the sum (\ref{power-sums}) can be obtained as the value of the corresponding polynomial $\hat{S} _m(n)$.
The exponential generating  function (e.g.f.) for sums of the form (\ref{power-sums}) is given by
\begin{eqnarray}
G(n, t)&=&\sum_{j\geq 0}S_j(n)\frac{t^j}{j!}=\sum_{j\geq 0}\left(\sum_{q=1}^{n}q^j\frac{t^j}{j!}\right)=\sum_{q=1}^{n}e^{qt}\nonumber\\
&=&\frac{e^{(n+1)t}-e^t}{e^t-1}.
\label{genf}
\end{eqnarray}
In turn, the expansion (\ref{genf}) into an infinite series
\[
G(n, t)=\sum_{q\geq 0}\hat{S}_{q}(n)\frac{t^q}{q!}
\]
gives an infinite  set of polynomials $\{\hat{S}_m (n): m\geq 0\} $ whose values for natural  $n$ give the values of the corresponding sums. In 1713, Jacob Bernoulli published  an expression for the sum of the degree $m$ of the first $n$ natural numbers as a polynomial of $(m+1)$-th degree in  $n$, which is determined by some infinite set of rational numbers $\{B_m: m\geq 0\}$, which are now called the Bernoulli numbers.

We will write this expression in the following form:
\begin{equation}
\hat{S}_{m}(n)=\frac{1}{m+1}\sum_{q=0}^m(-1)^q{m+1\choose q}B_{q}n^{m+1-q}.
\label{22}
\end{equation}
Faulhaber's theorem, in turn, asserts that, for odd values of $m$, the polynomial (\ref{22}) is actually a polynomial in $w:=n(n+1)$. To a certain extent this is due to the fact that all odd Bernoulli numbers starting with the third one are zero \cite{Knut}. In turn, it is known that for any $m$, any polynomial (\ref{22}) can be written as some polynomial in $w$ multiplied by $2n+1$ \cite{Jacobi}.

In applications there can arise not only sums of the form (\ref{power-sums}), but also other sums involving the powers of natural numbers. For example, Johann Faulhaber considered the so-called $k$-fold sums $S^k_m(n)$ defined by the recurrence relation
\[
S^{k}_m(n)=\sum_{q=1}^{n}S^{k-1}_m(q),\;\; \forall k\geq 1
\]
starting from $S^{0}_m(n):=n^m$. When considering such generalizations, the question always arises about of polynomials that, like (\ref{22}), help to calculate the values of corresponding sums.

In this paper we consider some generalizations of power-sum (\ref{power-sums}) and  set the problem of finding the corresponding polynomials. In the next section, we define a some class of sums whose e.g.f. is the $k$-th power of the generating function (\ref{genf}), with $k$ being an arbitrary natural number. We denote these sums by the symbol $S^{(k)}_m(n) $ and call them power-sums  of higher order. One of our results is that we found an expression for the associated  polynomials similar to the formula (\ref{22}). These polynomials, in contrast to (\ref{22}), are determined by higher-order Bernoulli numbers. This is, although not obvious, is quite natural. In the next section we give some information about higher-order Bernoulli numbers and about the Stirling numbers of the second kind, which are also involved in the formula for these polynomials. In the second section, we also investigate some rational roots of polynomials corresponding to higher order power-sums.
In the third section, we discuss some generalizations of the sums (\ref{power-sums}) arising in applications, namely, a some class of multiple sums.

\section{ higher order power-sums}

\subsection{Definition of higher order power-sums}

Let us consider the power of the e.g.f.  (\ref{genf}):
\[
\left(G(n, t)\right)^k:=\sum_{q\geq 0}S_{q}^{(k)}(n)\frac{t^q}{q!}.
\]
We have
\begin{equation}
\left(G(n, t)\right)^k=\left(\sum_{q=1}^{n}e^{qt}\right)^k=\sum_{q=0}^{k(n-1)}{k\choose q}_ne^{(k+q)t},
\label{1}
\end{equation}
where the symbol ${k\choose q}_n$  stands for   polynomial coefficients  defined through the relation
\[
\left(\sum_{q=1}^{n}t^q\right)^k:=\sum_{q=0}^{k(n-1)}{k\choose q}_nt^{k+q}.
\]
It should be noted that the coefficients ${k\choose q}_n$, naturally generalizing the binomial coefficients ($n=2$), originated  from Abraham De Moivre and Leonhard Euler works \cite{Moivre}, \cite{Euler}.
Then they were rediscovered in the works \cite{Montel} and \cite{Tremblay} and later were studied in detail in the literature due to their good applicability. By direct computation we derive
\begin{proposition}
The expression  (\ref{1}) is the e.g.f. for sums of the form
\begin{equation}
S_{m}^{(k)}(n)=\sum_{q=0}^{k(n-1)}{k\choose q}_{n}\left(k+q\right)^m.
\label{sums} 
\end{equation}
\end{proposition}
It is evident, that these sums, by definition, are the result of successive binomial convolutions, that is,
\[
S_m^{(k)}(n)=\sum_{q=0}^{m} {m\choose q}S_{q}^{(k-1)}(n)S_{m-q}(n),\;\;
k\geq 2.
\]
It is also obvious that this is true for the corresponding polynomials $\hat{S}_m^{(k)}(n)$. The following property holds:
\begin{proposition} \label{proposition2.2}
The sums  $S_m^{(k)}(n)$ satisfy the recurrence relation
\begin{eqnarray}
&&\sum_{q=0}^{m}(-1)^{q}{m+k\choose q}S(m+k-q, k)S_q^{(r)}(n) \nonumber  \\
      &&\;\;\;\;\;=\frac{1}{{k\choose r}}\sum_{j=0}^{m}(-1)^j{m+k\choose m+k-r-j}S(m+k-r-j, k-r)S(r+j, r)n^{r+j}.						
			\label{id}
\end{eqnarray}
where $S(n, k)$ are the Stirling numbers of the second kind.
\end{proposition} 
\begin{proof} 
The proposition is proved with the help of standard arguments. Let us make the change of the argument of the e.g.f. (\ref{genf}):  $t\rightarrow -t$.  Obviously, the following relation holds
\begin{equation}
(-1)^r\left(\frac{e^{-nt}-1}{e^{t}-1}\right)^{r}=\sum_{q\geq 0}(-1)^qS^{(r)}_{q}(n)\frac{t^q}{q!}.
\label{id1}
\end{equation}
Multiplying both sides of (\ref{id1}) by $(e^t-1)^k$ and taking into account that
\[
(e^t-1)^k=k!\left(\sum_{q\geq 0} S(q, k)\frac{t^q}{q!}\right),\;\;
\forall k\geq 0,
\]
we obtain, as a result, the relation (\ref{id}).
\end{proof} 
As a special case of (\ref{id}), for $r=k$, we get
\[
\sum_{q=0}^{m}(-1)^{q}{m+k\choose q}S(m+k-q, k)S_q^{(k)}(n) =(-1)^mS(k+m, k)n^{m+k}.						
\]
In turn, if $k=1$, then the last relation turns into a well-known relation for the classical sums of powers \cite{Riordan}. The following theorem gives a some representation of the polynomials $\hat{S} _{m}^{(k)}(n) $.
\begin{theorem} \label{theorem1}
The polynomials $\hat{S}_{m}^{(k)}(n)$ can be written as
\begin{equation} 
\hat{S}_{m}^{(k)}(n)=\frac{1}{{m+k\choose k}}\sum_{q=0}^{m}(-1)^q{m+k\choose q}B_q^{(k)}S(m+k-q, k)n^{m+k-q},
\label{higher-polynomials}
\end{equation}
where $B_q^{(k)}$ are the Bernoulli numbers of higher order that are defined by the e.g.f.
\begin{equation}
\frac{t^k}{(e^t-1)^k}=\sum_{q\geq 0}B^{(k)}_q\frac{t^q}{q!}.
\label{Bern-high}
\end{equation}
\end{theorem} 
Comparing (\ref {higher-polynomials}) with the formula (\ref{22}), we see that the Bernoulli numbers here are replaced by their higher analogs. Moreover we see that this expression, in contrast with (\ref{22}), involves the Stirling numbers of the second kind.

\subsection{Bernoulli numbers of higher order}

The classical Bernoulli numbers are known to be uniquely determined by the e.g.f. \cite{Graham}
\begin{equation}
\frac{t}{e^t-1}=\sum_{q\geq 0}B_q\frac{t^q}{q!}.
\label{Bernoulli}
\end{equation}
It follows from (\ref{Bernoulli}) the recurrence relation
\begin{equation}
\sum_{q=0}^{m}{m+1\choose q}B_{q}=\delta_{0, m}.
\label{rec-rel}
\end{equation}
This relation, in fact, is one of the many recurrence relations for the Bernoulli numbers (see, for example, \cite{Agoh}). For example, one can derive a countable set of recurrence relations of the form
\begin{equation}
\sum_{q=0}^{m}{m+k\choose q}S(m+k-q, k)B_{q}=\frac{m+k}{k}S(m+k-1, k-1),\;\;
\forall k\geq 1.
\label{rec-rel1}
\end{equation}
The numbers $B_n^{(k)} $ first appeared in the work of N\"orlund \cite{Norlund} in connection with the theory of finite differences and were subsequently investigated from various points of view (see, for example, \cite{Carlitz}). It is known that they are related to each other by the recurrence relation \cite{Norlund}
\begin{equation}
B_n^{(k+1)}=\frac{k-n}{k}B_n^{(k)}-nB_{n-1}^{(k)}.
\label{Norlund}
\end{equation}
It should be noted that $B^{(k)}_n$ for some fixed $n\geq 0$ is calculated as the value of some polynomial in $k$. In what follows we shall use for these polynomials the same notation $B^{(k)}_n$ in the hope that this will not lead to confusion. In the literature, polynomials of this type are called the N\"orlund polynomials. The first six of them are as follows
\[
B_0^{(k)}=1,\;\;
B_1^{(k)}=-\frac{1}{2}k,\;\;
B_2^{(k)}=\frac{1}{12}k\left(3k-1\right),\;\;
B_3^{(k)}=-\frac{1}{8}k^2\left(k-1\right),
\]
\[
B_4^{(k)}=\frac{k}{240}\left(15k^3-30k^2+5k+2\right),\;\;
B_5^{(k)}=-\frac{1}{96}k^2\left(k-1\right)\left(3k^2-7k-2\right).
\]
\begin{remark}
By standard arguments, using the e.g.f. (\ref{Bern-high}), we can derive the following recurrence relation:
\begin{equation}
\sum_{q=0}^{m}{m+k\choose q}S(m+k-q, k)B_{q}^{(r)}=\frac{{m+k\choose k}}{{m+k-r\choose k-r}}S(m+k-r, k-r),\;\; \forall k\geq r,
\label{impl1}
\end{equation}
the particular case of which is (\ref{rec-rel1}). Making use (\ref{impl1}), we can prove that polynomials defined by the formula (\ref{higher-polynomials}) satisfy the recurrence relation
(\ref{id}), but this is not necessary here.
\end{remark}

\subsection{Stirling numbers of the second kind}

The numbers $S(n, k)$ in the formula (\ref{higher-polynomials}), as mentioned above, are Stirling numbers of the second kind. As is known, they satisfy the identity
\begin{equation}
S(n, k)=S(n-1, k-1)+kS(n-1, k)
\label{rr}
\end{equation}
and some boundary conditions \cite{Graham}. The Stirling numbers $S(m + k, k)$ for some fixed $m\geq 0$ can be calculated as the values of some polynomial $f_m (k)$ of degree $2m$. These polynomials satisfy the recurrence identity
\begin{equation}
f_m(k)-f_m(k-1)=kf_{m-1}(k),
\label{Stirling}
\end{equation}
which easily follows from the identity (\ref{rr}). In the literature they are known as the Stirling polynomials \cite{Gessel}, \cite{Gessel1}, \cite{Jordan}. The comparison (\ref{Norlund}) and (\ref{Stirling}) shows that the N\"orlund and Stirling polynomials are related by the relation (see, for example, \cite{Adelberg})
\begin{equation}
f_m(k)={m+k\choose m}B_m^{(-k)}.
\label{Stirling1}
\end{equation}
Using (\ref{Stirling1}), we can rewrite (\ref{impl1}) in the form
\[
\sum_{q=0}^{m}{m\choose q}B_{m-q}^{(-k)}B_{q}^{(r)}=B_{m}^{(r-k)}.
\]
The last relation is completely obvious, since it is simply a overwriting of the relation $B^{r-k} = B^rB^{-k} $, where $B = B (t)$ is the e.g.f. of the Bernoulli numbers.

\subsection{Proof of the theorem \ref{theorem1}}

In the formula (\ref{higher-polynomials}) we can replace $S(m + k-q, k)$ by $f_ {m-q}(k)$ and get the following representation:
\begin{equation}
\hat{S}_{m}^{(k)}(n)=\frac{1}{{m+k\choose k}}\sum_{q=0}^{m}(-1)^q{m+k\choose q}{m-q+k\choose m-q}B_q^{(k)}B_{m-q}^{(-k)}n^{m-q+k}.
\label{18}
\end{equation}
In turn, using the binomial identity
\[
{m+k\choose q}{m-q+k\choose m-q}={m+k\choose k}{m\choose q},
\]
we obtain a simpler expression
\begin{equation}
\hat{S}_{m}^{(k)}(n)=n^k\sum_{q=0}^{m}(-1)^q{m\choose q}B_q^{(k)}B_{m-q}^{(-k)}n^{m-q}.
\label{19}
\end{equation}
The formula (\ref{19}) represents the binomial convolution of two sequences: $\{a_m:=(-1)^{m} B_m^{(k)}\}$ and $\{b_m:=B_m^{(- k)}n^{m + k}\}$. 
In turn, the e.g.f.'s of these sequences, as can be easily verified, are the $k$-degree of two e.g.f.'s
\begin{equation}
A(t)=\frac{-t}{e^{-t}-1}\;\;\mbox{and}\;\; C(n, t)=\frac{e^{nt}-1}{t}.
\label{AB}
\end{equation}
It is obvious that the representation of the e.g.f. (\ref{genf}) in the factorized form $G(n, t)=A(t)C(n, t)$ gives the simplest way to express the polynomials $\hat{S}_{m}(n)$ in terms of Bernoulli numbers. Since the expression (\ref{19}) is equivalent (\ref{18}), the theorem \ref{theorem1} should be considered proven.

\subsection{Some properties of the polynomials $\hat{S}_{m}^{(k)}(n)$}

In what follows we will consider these polynomials in the continuous variable $z\in\mathbb{C}$, so that we can talk about their derivatives. It is more convenient to investigate the polynomials\footnote{Obviously, $\hat{S}_{m}^{(k)}(z)=z^kQ_{m}^{(k)}(z)$.}
\[
Q_{m}^{(k)}(z)=\sum_{q=0}^{m}(-1)^q{m\choose q}B_q^{(k)}B_{m-q}^{(-k)}z^{m-q}.
\]
The following lemma will be useful for what follows.
\begin{lemma} \label{lemma1}
The following two identities are valid:
\begin{equation}
\sum_{q=0}^{m-1}{m\choose q}(m-q)B_q^{(k)}B_{m-q}^{(-k)}=\left\{ 
\begin{array}{l}
-kB_1,\;\;m= 1,\\[0.2cm]
kB_m,\;\;\forall m\geq 2
\end{array}
\right.
\label{191}
\end{equation}
and
\begin{equation}
\sum_{q=0}^{m-2}{m\choose q}(m-q)(m-q-1)B_q^{(k)}B_{m-q}^{(-k)}=\left\{ 
\begin{array}{l}
-k(3k-1)B_2
-2k^2B_{1},\;\;m= 2,\\[0.2cm]
-k((m+1)k-m+1)B_m
+mk^2B_{m-1},\;\;
\forall m\geq 3.
\end{array}
\right.
\label{1911}
\end{equation}
\end{lemma}
\begin{proof}
Let $B^{(k)}=B^{(k)}(t)$ be the e.g.f. of the Bernoulli numbers of higher order (\ref{Bern-high}). We have
\[
t\frac{d B^{(-k)}}{d t}B^{(k)}=kt\frac{e^t}{e^t-1}-k=k\left(A-1\right),
\]
where $A=A(t)$ is supposed to be the e.g.f. defined in (\ref{AB}). This relation implies (\ref{191}). To show (\ref{1911}), one needs to calculate
\begin{eqnarray}
t^2\frac{d^2 B^{(-k)}}{d t^2}B^{(k)}&=&\frac{t^2(k^2-k)e^{2t}}{(e^t-1)^2}+\frac{(t-2k)kte^t}{e^t-1}+k^2+k\nonumber\\
&=&k^2\left(A-1\right)^2-k\left(A^2-tA-1\right). \label{21}
\end{eqnarray}
Taking into account the recurrence relation (\ref{Norlund}), we derive from (\ref{21}) the relation (\ref{1911}).
\end{proof}
\begin{remark}
We could write (\ref{191}) as
\[
\sum_{q=0}^{m-1}{m\choose q}(m-q)B_q^{(k)}B_{m-q}^{(-k)}=(-1)^mkB_m,\;\;\forall m\geq 1.
\]
However, considering the fact that the Bernoulli numbers are nonzero only for even $m$ and $m=1$, we can get rid of the minus sign. A similar remark applies to the formula (\ref{1911}).
\end{remark}
Now we in a position to formulate some statements about the rational zeros of the polynomials $Q_{m}^{(k)}(z)$.
\begin{theorem} \label{2} The following two statements are true:
\begin{itemize}
\item[1)]
The polynomial $Q_{m}^{(k)}(z)$ has a simple root $z_0=-1$ for even $m$ and $m=1$ and any $k\geq 1$;
\item[2)]
The polynomial $Q_{m}^{(k)}(z)$ has a double root $z_0=-1$ for odd  $m\geq 3$ and any $ k\geq 1$.
\end{itemize}
\end{theorem}
\begin{proof}
It is obvious that
\[
Q_{m}^{(k)}(-1)=(-1)^{m}\sum_{q=0}^{m}{m\choose q}B_q^{(k)}B_{m-q}^{(-k)}=0,\;\;\forall m\geq 1,
\]
i.e, for any $m\geq 1$ and $k\geq 1$, the polynomial $Q_{m}^{(k)}(z)$ has the root $z_0=-1$. Further, we have
\[
\frac{d Q_{m}^{(k)}(z)}{dz}\left|_{z=-1} \right.=(-1)^{m-1}\sum_{q=0}^{m-1}{m\choose q}(m-q)B_q^{(k)}B_{m-q}^{(-k)}
\]
and by (\ref{191}), we obtain
\[
\frac{d Q_{m}^{(k)}(z)}{dz}\left|_{z=-1} \right.=
-kB_m,
\]
We know that the Bernoulli numbers $B_m$ are nonzero for even $m$ and $m=1$. Thus, the first part of the sentence (\ref{2}) is proved. For odd $m\geq 3$, the Bernoulli numbers are equal
zero and therefore the multiplicity of the root $z_0=-1 $ is at least equal to two. Next, we calculate
\[
\frac{d^2 Q_{m}^{(k)}(z)}{dz^2}\left|_{z=-1} \right.=(-1)^{m}\sum_{q=0}^{m-2}{m\choose q}(m-q)(m-q-1)B_q^{(k)}B_{m-q}^{(-k)}.
\]
Now we need to use the identity (\ref{1911}). By virtue of this identity, the second derivative $Q_{m}^{(k)}(z)$  at $z=-1$ for any $m\geq 2$ and $k\geq 1$ is not zero. Thus the theorem is proved.
\end{proof}
\begin{remark}
It was shown in the paper \cite{Kimura} that the rational zeros of the polynomials $\hat{S}_{m}^{(1)}(z)$ are only $0, -1$, and $-1/2$, and the last value is a simple root for these polynomials with even numbers.
\end{remark}
The polynomials $Q_{2}^{(k)}(z)$ and $Q_{3}^{(k)}(z)$, as is easily seen, have the following form:
\[
Q_{2}^{(k)}(z)=\frac{k}{12}(z+1)\left((3k+1)z+3k-1\right)\;\;\mbox{and}\;\;
Q_{3}^{(k)}(z)=\frac{k^2}{8}(z+1)^2\left((k+1)z+k-1\right).
\]
Hence it is obvious that
\[
Q_{2}^{(k)}\left(-\frac{3k-1}{3k+1}\right)=0,\;\;\forall k\geq 1\;\;\mbox{and}\;\;  Q_{3}^{(k)}\left(-\frac{k-1}{k+1}\right)=0,\;\;\forall k\geq 2.
\]

\section{Other generalizations  of power-sums}

\subsection{Multiple sums}

In the paper \cite{Svinin}, we considered multiple sums of the form
\begin{equation}
\tilde{\mathcal{S}}_{m}^{(k)}(n):=\sum_{\{q\}\in B_{k, kn}}\left(q_1^{m}+(q_2-n)^{m}+\cdots+(q_k-kn+n)^{m}\right),
\label{sums11} 
\end{equation}
where  the exponent $m$ is supposed to be odd. Here $B_{k, s}:=\{q_j : 1\leq q_1\leq \cdots \leq q_k\leq s\}$. Obviously, if $m$ is odd, then the value of $q^m$ with negative $q$ in (\ref{sums11}) is  $q^m = -|q|^m$. According to this rule, a sum of the form (\ref{sums11}) can be rewritten as
\begin{equation}
\tilde{\mathcal{S}}_{m}^{(k)}(n)=\sum_{q=1}^{kn}c_q(k, n)q^m
\label{defin}
\end{equation}
with some integer coefficients $c_q(k, n)$.
\begin{example}
In the case $k=2$, we have
\begin{eqnarray}
\tilde{\mathcal{S}}_{m}^{(2)}(n)&=&\sum_{\{q\}\in B_{2, 2n}}\left(q_1^{m}+(q_2-n)^{m}\right)\nonumber\\
                                &=&\sum_{j=1}^{n-1}\left(\sum_{q=1}^{j}q^m-j(n-j)^m\right)+\sum_{j=1}^{n}\left(\sum_{q=1}^{n+j}q^m+(n+j)j^m\right).\nonumber
\end{eqnarray}
Hence we get
\[
c_q(2, n)=\left\{
\begin{array}{l}
2n+q+1,\;\; 1\leq q\leq n, \\[0.2cm]
2n-q+1,\;\; n+1\leq q\leq 2n. 
\end{array}
\right.
\]
\end{example}

At first glance, multiple sums (\ref{sums11}) have nothing to do with  higher order power-sums  (\ref{sums}). However, in the paper \cite{Svinin}, it was conjectured that, in the case of odd $m$, the sum $\tilde{\mathcal{S}}_{m}^{(k)}(n)$ can be expressed  as\footnote{In this paper, the sums $S_{m}^{(k)}(n)$ were not identified as sums of higher orders.}
\begin{equation}
\tilde{\mathcal{S}}_{m}^{(k)}(n)=\sum_{q=0}^{k-1}{k(n+1)\choose q} S_{m}^{(k-q)}(n).
\label{relationsh}
\end{equation}
This hypothesis was based on numerous calculations using computer algebra method. Note that for $k=1$ the hypothesis is clear, and for $k=2$ it is easy to prove. By the way, if (\ref{relationsh}) is true, then it follows that the integer coefficients $c_q(k, n)$ are, in addition, non-zero and positive.

\subsection{Binomial sums}

On the right-hand side of (\ref{relationsh}) there is a sum, which we, somewhat conditionally, call binomial. Let us, throughout the  rest of the paper, study a sum of the form \cite{Svinin}
\begin{equation}
\mathcal{S}_{m}^{(k)}(n)=\sum_{q=0}^{k-1}{k(n+1)\choose q} S_{m}^{(k-q)}(n).
\label{sums1}
\end{equation}
In what follows, it is useful to consider its particular case
\[
\mathcal{S}_{m}^{(k)}(1)=\sum_{q=0}^{k-1}{2k\choose q} (k-q)^m.
\]
Observe that in the case of an odd exponent $m$, we can write
\begin{equation}
2\mathcal{S}_{m}^{(k)}(1)=\sum_{q=0}^{2k}{2k\choose q}|k-q|^m.
\label{23}
\end{equation}
On the right-hand side of this identity there is the binomial sum which was investigated in the papers \cite{Strazdins}, \cite{Tuenter}. Using the relation (\ref{23}), we can transfer some results of \cite{Tuenter} to the sum $\mathcal{S}_{m}^{(k)}(1)$.
\begin {proposition}
The sum $\mathcal{S}_{m}^{(k)}(1)$ satisfies the recurrence relation
\begin {equation}
\mathcal{S}_{m+2}^{(k)}(1)=k^2\mathcal{S}_{m}^{(k)}(1)-2k(2k-1)\mathcal{S}_{m}^{(k-1)}(1).
\label{241}
\end{equation}
\end{proposition}
\begin{proof}
We have
\begin{eqnarray}
k^2\mathcal{S}_{m}^{(k)}(1)-\mathcal{S}_{m+2}^{(k)}(1)&=&\sum_{q=1}^{k-1}{2k\choose q} (k-q)^m\left(k^2-(k-q)^2\right)\nonumber\\
&=&\sum_{q=1}^{k-1}{2k\choose q} (k-q)^mq\left(2k-q\right).\nonumber
\end{eqnarray}
Using the binomial identity
\[
{2k\choose q}={2k-2\choose q-1}\frac{2k(2k-1)}{q(2k-q)},
\]
we obtain, as a result, the relation (\ref{241}).
\end{proof}
\begin{proposition} There is an infinite set of polynomials $\{P_m(k): m\geq 0 \}$ such that,
\begin{equation}
\mathcal{S}_{2r+1}^{(k)}(1)=P_r(k)\frac{k}{2}{2k\choose k}.
\label{251} 
\end{equation}
\end{proposition}
\begin{proof} 
Let us first calculate $\mathcal{S}_{1}^{(k)}(1)$. We have
\begin{eqnarray}
\mathcal{S}_{1}^{(k)}(1)&=&\sum_{q=0}^{k}{2k\choose q} (k-q)=\frac{1}{2}\sum_{q=0}^{k}{2k\choose q} \left((2k-q)-q\right)\nonumber\\
&=&\frac{1}{2}\sum_{q=0}^{k}{2k\choose q} (2k-q)-\frac{1}{2}\sum_{q=1}^{k}{2k\choose q}q.\nonumber
\end{eqnarray}
Using binomial identities
\[
{2k\choose q}=\frac{2k}{2k-q}{2k-1\choose q}=\frac{2k}{q}{2k-1\choose q-1},
\]
we get
\[
\mathcal{S}_{1}^{(k)}(1)=k\sum_{q=0}^{k}{2k-1\choose q}-k\sum_{q=1}^{k}{2k-1\choose q-1}=k{2k-1\choose k}=\frac{k}{2}{2k\choose k}.
\]
Thus, $\mathcal{S}_{1}^{(k)}(1) $ is actually calculated by the formula (\ref{251}), with $P_0(k)=1$. Now, using the identity (\ref{241}) and the binomial identity
\[
2k(2k-1){2k-2\choose k-1}=k^2{2k\choose k},
\]
we see that if $\mathcal{S}_{2r+1}^{(k)}(1)$ is expressed by the formula (\ref{251}) with some polynomial $P_r(k)$, then $\mathcal{S}_{2r+3}^{(k)}(1)$ is also expressed by this expression with some polynomial $P_{r+1}(k)$, defined by 
\begin{equation}
P_{r+1}(k)=k^2P_{r}(k)-k(k-1)P_{r}(k-1).
\label{pol}
\end{equation}
Therefore, by the method of mathematical induction, we prove the proposition.
\end{proof} 
It is important to observe that the polynomials $P_{r}(k)$ are closely related to the well-known Gandhi polynomials \cite{Gandhi}, that are defined by the recurrence relation
\[
F_{r+1}(k)=(k+1)^2F_{r}(k+1)-k^2F_{r}(k),
\]
starting with $F_1(k)=1$. Indeed, it is not hard to see that $P_r(k)=(-1)^{r+1}kF_{r}(-k)$ for $r\geq 1$. The first six Gandhi polynomials are as follows:
\[
F_1(k)=1,\;\;
F_2(k)=2k+1,\;\;
F_3(k)=6k^2+8k+3,\;\;
F_4(k)=24k^3+60k^2+54k+17,
\]
\[
F_5(k)=120k^4+480k^3+762k^2+556k+155,
\]
\[
F_6(k)= 720k^5+4200k^4+10248k^3+12840k^2+8146k+2073.
\]
\begin{remark}
The Gandhi polynomials appeared in the paper \cite{Gandhi}, in which the hypothesis was formulated that
$F_r(0)=|G_{2r}|$, where $G_{2r}=2(1-4^r)B_{2r}$ are the alternating Genocci numbers determined by the e.g.f.
\[
\frac{2t}{e^t+1}=t+\sum_{q \geq 1}G_{2q}\frac{t^{2q}}{(2q)!}.
\] 
This hypothesis was proved in the papers \cite{Carlitz2}, \cite{Riordan1}. It should be noted that the Genocci numbers, unlike the Bernoulli numbers, are integers.
\end{remark}
\begin{remark}
The Gandhi polynomials can be written as $F_{r}(k)=F_r(1, 1, k)$, where $F_r(x, y, z)$ are some polynomials of three variables that are defined by the recurrence relation
\[
F_{r+1}(x, y, z)=(z+x)(z+y)F_{r}(x, y, z+1)-z^2F_{r}(x, y, z),
\]
starting with $F_1(x, y, z)=1$. In the literature, these polynomials are known as the Dumont-Foath polynomials \cite{Carlitz1}, \cite{Dumont}. It turns out that these polynomials are symmetric with respect to the variables $x, y, z$. The main property of these polynomials is that $F_r(1, 1, 1)=|G_{2r+2}|$.
\end{remark}

We observe now that  the formula (\ref{251}) can be rewritten in the form
\[
\mathcal{S}_{2r+1}^{(k)}(1)=(-1)^{r+1}F_r(-k)k^2{2k-1\choose k-1}
\]
for $r\geq 1$.

\subsection{Polynomials corresponding to binomial sums}

Let us define the set of polynomials $\{\hat{\mathcal{S}}_m^{(k)}(z): m\geq 1\} $ by the relation
\[
\hat{\mathcal{S}}_m^{(k)}(z)=\sum_{q=0}^{k-1}{k(z+1)\choose q} \hat{S}_{m}^{(k-q)}(z).
\]
Our goal is to describe these polynomials. For the degree $m=1$ we can formulate the following result:
\begin{proposition} \label{pr:3.2}
The polynomial $\hat{\mathcal{S}}_1^{(k)}(z)$ can be  defined by 
\[
\hat{\mathcal{S}}_1^{(k)}(z)=\frac{z(z+1)}{2}k{k(z+1)-1\choose k-1},\;\;
\forall k\geq 1.
\]
\end{proposition}
To prove this proposition, it suffices to use the following lemma:
\begin{lemma} \label{lemma2}
The relation
\begin{equation}
{k(z+1)-1\choose k-1}=\frac{1}{k}\sum_{q=0}^{k-1}(k-q){k(z+1)\choose q}z^{k-q-1}
\label{10}
\end{equation}
is an identity.
\end{lemma}
\begin{proof} 
Let us denote
\[
g_k(z):=k{k(z+1)-1\choose k-1}\;\;\mbox{and}\;\;\tilde{g}_k(z):=\sum_{q=0}^{k-1}(k-q){k(z+1)\choose q}z^{k-q-1}.
\]
Since $g_k(z)$ and $\tilde{g}_k(z)$ are polynomials in $z$ of order $k-1$, it is obvious that to prove the identity (\ref{10}, it suffices to show that
\[
g_k^{(r)}(0)=\tilde{g}_k^{(r)}(0),\;\; \forall r=0,\ldots, k-1.
\] 
For example, it is easy to see that $g_k(0)=\tilde{g}_k(0)=k$. Let $\xi:=k!$, while $\xi_{i_1,\ldots, i_m}$ denote the product $\prod_{i=1}^ki $, where there are no numbers $i_1<i_2<\ldots <i_m$ . With this notation, we can write
\begin{equation}
g_k^{(r)}(0)=\frac{r!k^r}{(k-1)!}\sum_{1\leq i_1<\cdots< i_r\leq k-1}\xi_{i_1, \ldots, i_r}.
\label{24}
\end{equation}
In turn, for $\tilde{g}_k(z)$, we get the following:
\begin{equation}
\tilde{g}_k^{(r)}(0)=r!\sum_{q=0}^{r}\frac{k^{r-q}}{q!(k-q-1)!}\sum_{q+1\leq i_1<\cdots< i_{r-q}\leq k-1}\xi_{i_1+1, \ldots, i_{r-q}+1}.
\label{25}
\end{equation}
Equating (\ref{24}) to (\ref{25}) and multiplying the resulting relation by $(k-1)!/r!$, we get the equality 
\begin{equation}
k^r\sum_{1\leq i_1<\cdots< i_r\leq k-1}\xi_{i_1, \ldots, i_r}=\sum_{q=0}^{r}{k-1\choose q}k^{r-q}\sum_{q+1\leq i_1<\cdots< i_{r-q}\leq k-1}\xi_{i_1+1, \ldots, i_{r-q}+1},
\label{i0}
\end{equation}
that remains for us to prove.

For convenience, let us denote
\[
D_{q}:=\{i_j : q+1\leq i_1<\cdots< i_{r-q}\leq k-1\}.
\] 
Obviously, the partition
\[
D_{0}=D^{(1)}_{0}\bigsqcup D^{(2)}_{0}
\]
with the parts
\[
D^{(1)}_{0}:=\left\{i_j : i_1=1;\; 2\leq i_2<\cdots< i_r\leq k-1\right\}
\]
and
\[
D^{(2)}_{0}:=\left\{i_j :  2\leq i_1<\cdots< i_r\leq k-1\right\}
\]
gives the relation
\begin{eqnarray}
\sum_{\{i_j\}\in D_{0}}\xi_{i_1, \ldots, i_r}&=&\sum_{2\leq i_2<\cdots< i_r\leq k-1}\xi_{1, i_2, \ldots, i_{r}}+\sum_{2\leq i_1<\cdots< i_r\leq k-1}\xi_{i_1, \ldots, i_{r}}\nonumber\\
&=&\sum_{\{i_j\}\in D_{1}}\xi_{i_1, \ldots, i_{r-1}}+\sum_{2\leq i_1<\cdots< i_r\leq k-1}\xi_{i_1, \ldots, i_{r}}.   \label{i1}
\end{eqnarray}
On the other hand, making  use  a partition
\[
D_{0}=\tilde{D}^{(1)}_{0}\bigsqcup \tilde{D}^{(2)}_{0}
\]
with
\[
\tilde{D}^{(1)}_{0}:=\left\{i_j : i_r=k-1;\; 1\leq i_1<\cdots< i_{r-1}\leq k-2\right\}
\]
and
\[
\tilde{D}^{(2)}_{0}:=\left\{i_j :  1\leq i_1<\cdots< i_r\leq k-2\right\},
\]
gives the relation
\begin{eqnarray}
\sum_{\{i_j\}\in D_{0}}\xi_{i_1+1, \ldots, i_r+1}&=&\sum_{1\leq i_1<\cdots< i_{r-1}\leq k-2}\xi_{i_1+1, \ldots, i_{r-1}+1, k}+\sum_{1\leq i_1<\cdots< i_r\leq k-2}\xi_{i_1+1, \ldots, i_{r}+1}\nonumber\\
&=&\frac{1}{k}\sum_{\{i_j\}\in D_{1}}\xi_{i_1, \ldots, i_{r-1}}+\sum_{2\leq i_1<\cdots< i_r\leq k-1}\xi_{i_1, \ldots, i_{r}}.
\label{i2}
\end{eqnarray}
Using (\ref{i1}) and (\ref{i2}), we can rewrite (\ref{i0}) as
\[
(k-1)k^{r-1}\sum_{\{i_j\}\in D_{1}}\xi_{i_1, \ldots, i_{r-1}}=\sum_{q=1}^{r}{k-1\choose q}k^{r-q}\sum_{\{i_j\}\in D_{q}}\xi_{i_1+1, \ldots, i_{r-q}+1}.
\]
Further, using similar partitionings of the sets $D_1, D_2,\ldots$, we continue step-by-step the transformations of the initial relation. At the $m$-th step we obtain this relation in the form
\[
{k-1\choose m}k^{r-m}\sum_{\{i_j\}\in D_{m}}\xi_{i_1, \ldots, i_{r-m}}=\sum_{q=m}^{r}{k-1\choose q}k^{r-q}\sum_{\{i_j\}\in D_{q}}\xi_{i_1+1, \ldots, i_{r-q}+1}.
\]
It is obvious that on the $r$-th step we obtain a trivial identity. Thus, the lemma is proved.
\end{proof} 
Taking into account that $S^{(k)}_1(z)=k(z+1)z^k/2$, from the lemma \ref{lemma2} we obtain the proposition \ref{pr:3.2}. 
This proposition  is consistent with the assumption made in the paper \cite{Svinin}, according to which the polynomials $\hat{\mathcal{S}}_{2r+1}^{(k)}(z)$, for $r\geq 0$ are determined by\footnote{Here this formula is written somewhat differently than in the work \cite{Svinin}.}
\begin{equation}
\hat{\mathcal{S}}_{2r+1}^{(k)}(z)=\left(-\frac{w}{2}\right)^{r+1}F_r(w, -k)k^2{k(z+1)-1\choose k-1},
\label{36}
\end{equation}
where $F_r(w, k)$ are certain polynomials in $k$ of degree $r-1$ with coefficients that are Laurent polynomials in the variable $w:=z(z+1) $. In particular, for $z=1$, we have
\[
\hat{\mathcal{S}}_{2r+1}^{(k)}(1)=(-1)^{r+1}F_r(2, -k)k^2{2k-1\choose k-1}=(-1)^{r+1}F_r(2, -k)\frac{k^2}{2}{2k\choose k}.
\]
The last expression agrees with (\ref{251}). 

The first six polynomials $F_r(w, k)$ in $k $ are defined as follows: \cite{Svinin}:
\[
F_1(w, k)=1,\;\;
F_2(w, k)=2k+\frac{2}{3}\frac{w+1}{w},\;\;
\]
\[
F_3(w, k)=6k^2+\frac{16}{3}\frac{w+1}{w}k+\frac{4}{3}\frac{(w+1)^2}{w^2},\;\; 
\]
\[
F_4(w, k)= 24k^3+40\frac{(w+1)}{w}k^2+24\frac{(w+1)^2}{w^2}k+\frac{24}{5}\frac{(w+1)^3}{w^3}+\frac{8}{5}\frac{1}{w},
\]
\begin{eqnarray}
F_5(w, k)&=& 120k^4+320\frac{(w+1)}{w}k^3+\frac{1016}{3}\frac{(w+1)^2}{w^2}k^2\nonumber\\
    &&+\left(160\frac{(w+1)^3}{w^3}+32\frac{1}{w}\right)k+\frac{80}{3}\left(\frac{(w+1)^4}{w^4}+\frac{w+1}{w^2}\right)\nonumber
\end{eqnarray}
\begin{eqnarray}
F_6(w, k)&=& 720k^5+2800\frac{(w+1)}{w}k^4+\frac{13664}{3}\frac{(w+1)^2}{w^2}k^3\nonumber\\
&&+\left(\frac{55936}{15}\frac{(w+1)^3}{w^3}+\frac{2544}{5}\frac{1}{w}\right)k^2+\left(\frac{22112}{15}\frac{(w+1)^4}{w^4}+\frac{13664}{15}\frac{(w+1)}{w^2}\right)k\nonumber\\
&&+\frac{22112}{105}\frac{(w+1)^5}{w^5}+\frac{44224}{105}\frac{(w+1)^2}{w^3}.\nonumber
\end{eqnarray}
A direct check shows that for these examples, $F_r(2, k)=F_r(k)$ is true, where the polynomials $F_r(k)$ are the Gandhi polynomials. The examples written out above suggest that polynomials
$F_r(w, k)$ should be written in the form
\[
F_r(w, k)=\sum_{q\geq 0}\frac{1}{w^q}\sum_{j=0}^{r-3q-1}F^{(q)}_{r,j}\frac{(w+1)^{r-j-3q-1}}{w^{r-j-3q-1}}k^j
\]
with some coefficients $ F^{(q)}_{r, j}$. For any fixed $r$, the sum over $q$ is in fact finite, since, by convention, if $r-3q-1\leq -1 $, then the sum over $j$ is assumed to be zero.

\end{document}